\documentclass[11]{amsart}

\usepackage{amsfonts}
\usepackage{amscd}
\usepackage{amsmath}
\usepackage{amsthm}
\usepackage{setspace}
\usepackage{hyperref}
\usepackage{color}
\usepackage{epsfig}
\usepackage{graphicx}
\usepackage{psfrag}
\usepackage{graphicx,transparent}
\usepackage{enumerate}
\usepackage{enumitem}

\theoremstyle{plain}
\newtheorem{theorem}{Theorem}[section]
\newtheorem{lemma}[theorem]{Lemma}
\newtheorem{proposition}[theorem]{Proposition}

\theoremstyle{definition}
\newtheorem{remark}[theorem]{Remark}

\newtheorem{example}[theorem]{Example}

\newcommand{\MM}{\mathcal M}
\newcommand{\BM}{\overline{\mathcal M}}

\newcommand{\CC}{\mathcal C}
\newcommand{\EE}{\mathcal E}

\newcommand{\OO}{\mathcal O}

\newcommand{\BEff}{\overline{\operatorname{Eff}}}

\newcommand{\wt}{\widetilde}

\newcommand{\irr}{\operatorname{irr}}

\newcommand{\Pic}{\operatorname{Pic}}
\newcommand{\Bl}{\operatorname{Bl}}

\newcommand{\bbF}{\mathbb F}

\newcommand{\bbP}{\mathbb P}

\newcommand{\bbZ}{\mathbb Z}

\title{Extremal effective divisors of Brill-Noether and Gieseker-Petri type in $\BM_{1,n}$}

\date{\today}

\author{Dawei Chen}

\address{Department of Mathematics, Boston College, Chestnut Hill, MA 02467}

\author{Anand Patel}

\email{dawei.chen@bc.edu, anand.patel@bc.edu}

\thanks{During the preparation of this article the first author was partially supported by the NSF grant DMS-1200329
and the NSF CAREER grant DMS-1350396.}

\begin{document}

\begin{abstract}
We show that certain divisors of Brill-Noether and Gieseker-Petri type span extremal rays of the effective cone in the moduli space of stable genus one curves with $n$ ordered marked points. In particular, they are different from the infinitely many extremal rays found in \cite{ChenCoskun}. 
\end{abstract}

\maketitle

\setcounter{tocdepth}{1}

\tableofcontents

\section{Introduction}
\label{sec:intro}

Let $\BM_{g,n}$ be the Deligne-Mumford moduli space of stable genus $g$ curves with $n$ ordered marked points. Denote by $\BEff(\BM_{g,n})$ the cone 
of pseudoeffective divisors on $\BM_{g,n}$. Understanding the structure of $\BEff(\BM_{g,n})$ plays a central role in the study of the birational geometry 
of $\BM_{g,n}$, see e.g.~\cite{HarrisMumfordKodaira, HarrisKodaira, EisenbudHarrisKodaira, FarkasKoszul, LoganKodaira, Vermeire, CastravetTevelev}. 

In \cite[Theorem 1.1]{ChenCoskun}, it was shown that there exist infinitely many extremal effective divisors in $\BM_{1,n}$ for each $n\geq 3$. It provides the first (and the only) known example of $\BM_{g,n}$ whose pseudoeffective cone is \emph{not} finitely generated. Recall the definition of those divisors. Let ${\bf a} = (a_1, \ldots, a_n)$ be a collection of $n$ integers satisfying that $\sum_{i=1}^n a_i = 0$, not all equal to zero. Define $D_{\bf a}$ in $\BM_{1,n}$ as the closure of the divisorial locus parameterizing smooth genus one curves with $n$ ordered marked points $(E; p_1, \ldots, p_n)$ such that $\sum_{i=1}^n a_i p_i \sim 0$ in $E$. For $n\geq 3$ and $\gcd (a_1, \ldots, a_n) = 1$, $D_{\bf a}$ spans an extremal ray 
of $\BEff(\BM_{1,n})$. 

A natural question is \emph{whether the $D_{\bf a}$ (and the boundary components) span all (rational) extremal rays of $\BEff(\BM_{1,n})$}? This is a meaningful question and one might expect an affirmative answer, as we explain in the following example. Consider the abelian surface 
$E\times E$, where $E$ is a general smooth elliptic curve with $p_0$ as the origin. Take $a_1, a_2 \in \bbZ$ such that they are relatively prime. Consider the locus  
$$C = \{(p_1, p_2)\in E\times E\ | \ a_1 p_1 + a_2 p_2 \sim (a_1+a_2)p_0 \}.$$ 
We know that $C$ spans an extremal ray 
of $\BEff(E\times E)$. Moreover, all (rational) extremal rays of $\BEff(E\times E)$ are spanned by such $C$, see \cite[II 4.16]{Kollar}. 
Note that $C$ is an analogue of $D_{\bf a}$ with ${\bf a} = (a_1, a_2, -a_1 - a_2)$ when we fix the moduli of a genus one curve with three marked points. 

Nevertheless, the main result of this paper shows that the above question has a negative answer. 

\begin{theorem}
\label{thm:main}
For every $n\geq 6$, there exist extremal effective divisors in $\BM_{1,n}$ that are different from the $D_{\bf a}$'s. 
\end{theorem}

See Theorems~\ref{thm:trigonal-extremal}, \ref{thm:d-gonal} and \ref{thm:GP} for a more precise statement. 

Let us explain our method. For a stable genus one curve with $2m$ marked points $(E; p_1, \ldots, p_{2m})$, identify 
$p_{2i-1}$ and $p_{2i}$ as a node for $i= 1, \ldots, m$. We thus obtain an $m$-nodal curve of 
arithmetic genus $m+1$. It induces a gluing morphism $\pi: \BM_{1, 2m}\to \BM_{m+1}$. We will show that pulling back certain divisors of Brill-Noether and Gieseker-Petri type by $\pi$ gives rise to extremal effective divisors 
different from the $D_{\bf a}$'s. To verify extremality, we exhibit a moving curve in (the main component of) the pullback divisor such that the curve has negative intersection number with the divisor. This idea was also used in \cite{OpieExtremal} to construct extremal effective divisors in $\BM_{0,n}$ that are different from the hypertree divisors in \cite{CastravetTevelev}.

The paper is organized as follows. In Section~\ref{sec:prelim}, we review basic divisor theory of $\BM_{g,n}$ and carry out the calculation of 
pulling back divisor classes under the gluing map $\pi: \BM_{1,2m}\to \BM_{m+1}$. In Section~\ref{sec:trigonal}, we verify the extremality of the main component of the pullback 
of the Brill-Noether trigonal divisor. In Section~\ref{sec:d-gonal}, we study the pullback of Brill-Noether $d$-gonal divisors for general $d$ and show that their main components are extremal. Finally, in Section~\ref{sec:GP} we study the pullback of the Gieseker-Petri divisor from $\BM_{4}$ and show that its main component is extremal. 

\subsection*{Acknowledgements} We would like to thank Gabriel Bujokas, Ana-Maria Castravet, 
Izzet Coskun, Anand Deopurkar, Maksym Fedorchuk, Joe Harris and Jenia Tevelev for valuable discussions related to this paper. 

\section{Preliminaries on moduli spaces of curves}
\label{sec:prelim}

Denote by $\lambda$ the first Chern class of the Hodge bundle on $\BM_{g,n}$. Let $\Delta_{\irr}$ be the locus in $\BM_{g,n}$ parameterizing curves with a non-separating node. 
For $0\leq i \leq [g/2]$, $S\subset \{1,\ldots, n \}$ and $2i -2 + |S| \geq 0$, let $\Delta_{i;S}$ denote the closure of the locus in $\BM_{g,n}$ that parameterizes nodal curves consisting of two components of genera $i$ and $g-i$, respectively, where
the genus $i$ component contains the marked points labeled by $S$. Denote by $\delta_{\bullet}$ the divisor class of $\Delta_{\bullet}$ and  
let $\delta$ be the class of the union of all boundary divisors on $\BM_{g,n}$. 
Let $\psi_i$ be the first Chern class of the cotangent line bundle on $\BM_{g,n}$ associated to the $i$th marked point for $1\leq i \leq n$. 
These divisor classes are defined on the moduli stack instead of the coarse moduli scheme, see e.g.~\cite{ArbarelloCornalba, HarrisMorrison} for more details.

In this paper, we focus on $\BM_{1,n}$ and $\BM_g$. The rational Picard group of 
$\BM_g$ is generated by $\lambda, \delta_{\irr}, \delta_1, \ldots, \delta_{[g/2]}$, and for $g\geq 3$ these divisor classes 
form a basis. The rational Picard group of $\BM_{1,n}$ has a basis given by $\lambda$ and $\delta_{0;S}$ for $|S| \geq 2$. The divisor classes $\delta_{\irr}$ and $\psi_i$ on $\BM_{1,n}$
can be expressed as 
\begin{eqnarray}
\label{eqn:delta}
\delta_{\irr} = 12 \lambda, 
\end{eqnarray}
\begin{eqnarray}
\label{eqn:psi}
\psi_i = \lambda + \sum_{i\in S} \delta_{0;S}.
\end{eqnarray}
Since $\delta_{\irr}$ and $\lambda$ are proportional on $\BM_{1,n}$, we will use them interchangeably throughout the paper.  
 
For a stable genus one curve with $2m$ marked point $(E; p_1, \ldots, p_{2m})$, identify 
$p_{2i-1}$ and $p_{2i}$ as a node for $i= 1, \ldots, m$. We thus obtain a curve of 
arithmetic genus $m+1$ with $m$ non-separating nodes. This induces a morphism 
$$\pi: \BM_{1, 2m}\to \BM_{m+1}. $$
The image of $\pi$ is contained in $\Delta_{\irr}$. 

Let us calculate the pullback of divisor classes from $\BM_{m+1}$ to $\BM_{1,2m}$ via $\pi$. For $1\leq i \leq m$, define
$$\Lambda_i = \{ S\subset \{1,\ldots,2m\} \ | \ S = \{2k_1-1, 2k_1, \ldots, 2k_i-1, 2k_i\}, 1\leq k_1 < \cdots < k_i \leq m \}. $$
In other words, for $S\in \Lambda$, $S$ contains the labeling of $i$ pairs of marked points that are glued to $i$ nodes. 
Denote by $S^c$ the complement of $S$ in $\{ 1, \ldots, n  \}$. 

\begin{proposition}
\label{prop:pullback}
Under the above setting, we have 
$$ \pi^{*}\lambda = \lambda, $$
$$ \pi^{*}\delta_i = \sum_{S\in \Lambda_i} \delta_{0; S} + \sum_{S^c\in \Lambda_{i-1}} \delta_{0; S}, \quad i <\frac{m+1}{2}, $$
$$ \pi^{*}\delta_i = \sum_{S\in \Lambda_i} \delta_{0; S}, \quad i = \frac{m+1}{2} \ \mbox{for odd}\ m, $$
$$ \pi^{*}\delta = (12 - 2m)\lambda - \sum_{|S|\geq 2} (|S| - 1)\delta_{0;S}. $$
\end{proposition}

\begin{proof}
We can calculate $\pi^{*}\lambda$ by induction on the number of pairs of points that are glued to a node. Take a family $f : \CC \to B$ of stable curves of genus $g-1$ with two disjoint sections $P_1$ and $P_2$. Consider the exact sequence 
$$ 0\to \Omega \to \Omega(P_1 + P_2) \to \OO_{P_1}\oplus \OO_{P_2} \to 0, $$
where $\Omega$ is the relative dualizing sheaf of the family $f: \CC \to B$. 
Applying $f_*$, we obtain that 
$$ 0\to f_{*}\Omega \to f_{*}(\Omega(P_1 + P_2)) \to \OO_{B}\oplus \OO_B\to \OO_B \to 0. $$
It follows that   
$$  \pi^{*}\lambda = c_1(f_{*}(\Omega(P_1 + P_2))) = c_1(f_{*}\Omega) = \lambda. $$

If the image of $(E; p_1, \ldots, p_{2m})\in \Delta_{0;S}$ under $\pi$ lies in $\Delta_i$, 
either the rational tail or the genus one tail  becomes a component of arithmetic genus $i$ after the gluing process, which corresponds to the case $S\in \Lambda_i$ or $S^c \in \Lambda_{i-1}$, respectively. Note that for $1\leq i \leq \frac{m+1}{2}$, $S\in \Lambda_i$ and $S^c \in \Lambda_{i-1}$ hold simultaneously if and only if $m$ is odd and $ i = \frac{m+1}{2}$. 
The two equalities about $\pi^{*}\delta_i$ follow right away. 

Finally, let $f: \EE\to B$ be a family of stable genus one curves with $2m$ sections $P_1, \ldots, P_{2m}$. 
If a node is obtained by identifying $P_{2k-1}$ and $P_{2k}$, its contribution in $\pi^{*}\delta$ is 
$$f_{*}(P_{2k-1}^2) + f_{*}(P_{2k}^2) = - \psi_{2k-1} - \psi_{2k}. $$
If a fiber $E$ in $\EE$ is contained in $\delta_{0;S}$ or $\delta_{\irr}$, a node of $E$ remains to be a node after the gluing process. This implies that 
$$ \pi^{*}\delta  =  \sum_{i=1}^{2m} (-\psi_i) + \delta = (12 - 2m)\lambda - \sum_{|S|\geq 2} (|S| - 1)\delta_{0;S}, $$
 where the relations \eqref{eqn:delta} and \eqref{eqn:psi} are used in the last step.  
\end{proof}

\section{Pulling back the trigonal divisor}
\label{sec:trigonal}

Consider  $\pi: \BM_{1,8}\to \BM_5$. Let $BN^1_3\subset \BM_5$ denote the Brill-Noether trigonal divisor whose general point parameterizes a curve admitting a triple cover of $\bbP^1$. 
By \cite{HarrisMumfordKodaira}, it has divisor class  
\begin{eqnarray*}
BN^1_3 & = & 8 \lambda - \delta_{\irr} - 4\delta_1 - 6 \delta_2 \\
       & = & 8 \lambda - \delta - 3\delta_1 - 5 \delta_2.   
\end{eqnarray*}
By Proposition~\ref{prop:pullback}, we obtain that  
\begin{eqnarray}
\label{eqn:trigonal}
\pi^{*}BN^1_3  =  4\lambda + \sum_{S\not\in \cup \Lambda_i} (|S| - 1)\delta_{0;S} + 4 \delta_{0; \{1, \ldots, 8\}}  \\
  - 2\sum_{k=1}^4 \delta_{0; \{2k-1, 2k \}} - 2\sum_{i<j}  \delta_{0; \{2i-1, 2i, 2j-1, 2j \}} \nonumber. 
\end{eqnarray}

Note that $\pi^{-1}BN^1_3$ may contain some boundary components of $\BM_{1,8}$, see Remark~\ref{rem:trigonal-component}. Denote by $\wt{BN}^1_3$ the main component of $\pi^{-1}BN^1_3$, i.e. 
$\wt{BN}^1_3$ is actually the closure of $(\pi^{-1}BN^1_3)\cap \MM_{1,8}$. (See Lemma~\ref{lem:d-gonal-irreducible} 
for the irreducibility of $(\pi^{-1}BN^1_d)\cap \MM_{1,4d-4}$ for general $d$.) 

Let us characterize when $(E; p_1,\ldots, p_8)$ is contained 
in $\wt{BN}^1_3$ for a smooth genus one curve $E$ with eight distinct marked points. 

\begin{proposition}
\label{prop:trigonal-geometry}
In the above setting, $(E; p_1,\ldots, p_8)$ is a general point of $\wt{BN}^1_3$ if and only if $E$ admits an embedding as a plane cubic in $\bbP^2$ such that the four lines 
$\overline{p_{1}p_{2}}$, $\overline{p_{3}p_{4}}$, $\overline{p_{5}p_{6}}$ and $\overline{p_{7}p_{8}}$ are concurrent. 
\end{proposition}

\begin{proof}
Identifying $p_{2k-1}, p_{2k}$ in $E$ for $k = 1, 2, 3, 4$, we obtain a curve $C$ of arithmetic genus five with four nodes. 
Suppose that $C$ is a trigonal curve. Then the canonical model of $C$ is contained in a cubic surface in $\bbP^4$. Perform elementary transformations at the nodes of $C$ and blow down the proper transform of the hyperplane section $H$ spanned by the nodes. The cubic surface is transformed to $\bbP^2$ and the image of $C$ is  
the embedding of $E$ as a plane cubic in $\bbP^2$. The exceptional curve containing $p_{2k-1}, p_{2k}$ is transformed to a line spanned by $p_{2k-1}, p_{2k}$ 
for $k = 1, 2, 3, 4$. The four lines are concurrent at a point $v$ arising from the contraction image of $H$. 

Conversely, if $(E; p_1, \ldots, p_8)$ admits such a plane cubic configuration, 
projecting $E$ from $v$ to a line gives rise to a $g^1_3$ on $E$, which descends to a $g^1_3$ on the $4$-nodal curve $\pi(E)$, because $p_{2k-1}$ and $p_{2k}$ map to the same image under the $g^1_3$ for $k=1,2,3,4$. 
\end{proof}

Next we construct a curve $B$ moving in $\wt{BN}^1_3$ such that $B\cdot \wt{BN}^1_3 < 0$. Fix four general concurrent lines $L_1, \ldots, L_4$ in $\bbP^2$ and fix two general points $p_{2k-1}, r_k \in L_k$ for each $k$. Consider 
the pencil $B$ of plane cubics passing through the eight fixed points. Denote by $p_{2k}$ the remaining intersection point of the cubics with $L_k$ for each $k$. See Figure~\ref{fig:cubic} for this configuration. 

\begin{figure}[htb]
    \centering
    \def\svgwidth{200pt}
     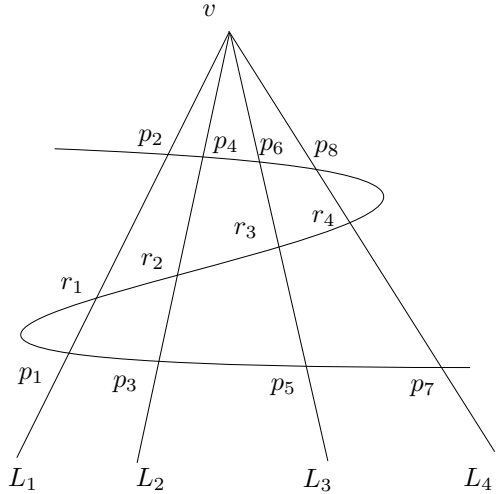
        \caption{\label{fig:cubic} A plane cubic marked at $p_1,\ldots, p_8$ }
    \end{figure}

Marking $p_1,\ldots, p_8$, $B$ can be viewed as a family of genus one curves with eight ordered marked points. By Proposition~\ref{prop:trigonal-geometry}, $B$ is a moving curve in $\wt{BN}^1_3$.  

\begin{lemma}
\label{lem:trigonal-intersection}
On $\BM_{1,8}$ we have the following intersection numbers: 
$$ B\cdot \lambda = 1, $$
$$ B\cdot \delta_{0; \{2k-1, 2k \}} = 1, $$
$$ B\cdot \delta_{0; \{2,4,6,8\}} = 1, $$
$$ B\cdot \delta_{0; S} = 0, \quad S\neq \{ 2k-1, 2k \}, \{2,4,6,8\}. $$
\end{lemma}

\begin{proof}
Since $p_{2k}$ varies in $L_k$, when it coincides with the fixed $p_{2k-1}$, we obtain a curve parameterized in $\delta_{0; \{2k-1, 2k \}}$. There is no reducible cubic that can pass through the eight fixed points, hence every curve parameterized by $B$ is an irreducible genus one curve. Moreover, since the locus of cuspidal cubics has codimension two in the total space of plane cubics, it implies that for a general configuration as in Figure~\ref{fig:cubic}, the pencil $B$ does not contain cuspidal cubics. 
Finally, if the cubic passes through the common point $v$ of the $L_k$, then $p_2, p_4, p_6, p_8$ coincide. Since they approach $v$ in different directions,  
blowing up $v$ results in a stable curve contained in $\delta_{0; \{2,4,6,8\}}$. The desired intersection numbers follow right away. 
\end{proof}

Now we can show that $\wt{BN}^1_3$ is an extremal effective divisor. Denote by $f: \BM_{1, n} \to \BM_{1, m}$ the forgetful morphism forgetting the last $n-m$ marked points.

\begin{theorem}
\label{thm:trigonal-extremal}
The divisor $\wt{BN}^1_3$ is an extremal divisor on $\BM_{1,8}$. For every $n\geq 8$, $f^{*}\wt{BN^1_8}$ spans an extremal ray of 
$\BEff(\BM_{1,n})$ that is different from the ray of $D_{\bf a}$. 
\end{theorem}

\begin{proof}
Write $\pi^{*}BN^1_3 = \wt{BN}^1_3 + \overline{BN}^1_3$, where $\overline{BN}^1_3$ is the union of boundary divisors contained in $\pi^{-1}BN^1_3$. 
Using Lemma~\ref{lem:trigonal-intersection} and \eqref{eqn:trigonal}, a routine calculation shows that $B\cdot \pi^{*}BN^1_3 = -1 <  0$. Since 
$B$ is not entirely contained in the boundary of $\BM_{1,8}$, it implies that $B\cdot \overline{BN}^1_3 \geq 0$, hence 
$B\cdot  \wt{BN}^1_3 < 0$. By \cite[Lemma 4.1]{ChenCoskun}, we know that $ \wt{BN}^1_3$ is extremal and rigid on $\BM_{1,8}$. 

For $n\geq 8$, take $n-8$ general very ample divisors $H_i$ on $\BM_{1,n}$ and use them to cut out a curve $B'$ 
in $f^{-1}B$. Then the class of $f_{*}B'$ is a positive multiple of $B$. By the projection formula, 
$B' \cdot f^{*}\wt{BN}^1_3 = (f_{*}B')\cdot \wt{BN}^1_3 < 0$. 
Moreover, varying $B$ in $\wt{BN}^1_3$ and $H_i$ in $\BM_{1,n}$, it follows that $B'$ is a moving curve in 
$f^{-1}\wt{BN}^1_3$. Therefore,  $f^{*}\wt{BN}^1_3$ is extremal and rigid on $\BM_{1,n}$. 
 
Recall that the divisor $D_{\bf a} = D_{\bf -a}$ on $\BM_{1,n}$ parameterizes $(E; p_1, \ldots, p_n)$ 
where $\sum_{i=1}^n a_i p_i \sim 0$ in $E$, $\sum_{i=1}^n a_i = 0$ and $\gcd(a_1,\ldots, a_n) = 1$. 
Since $D_{\bf a}$ and $f^{-1}\wt{BN}^1_3$ are rigid, in order to prove that they span different extremal rays, it suffices to show that they have different supports. If $f^{-1}\wt{BN}^1_3$ and $D_{\bf a}$ are set-theoretically the same, then $a_i = 0$ for $ i > 8$ since there is no constraint imposed 
to $p_i$ for $i > 8$ in the definition of $\wt{BN}^1_3$. Moreover, by the symmetry between the four pairs of nodes $\{ p_{2k-1}, p_{2k}\}$ , we conclude that $\{a_{2k-1}, a_{2k}\} = \{ c, c\}$ or $\{ -c, -c\}$ for $1\leq k\leq 4$. It follows that $c = 1$ and without loss of generality, say, 
$a_1 = a_2 = a_3 = a_4 = 1$, $a_5 = a_6 = a_7 = a_8 = -1$ (up to reordering the four nodes). Nevertheless, 
the resulting relation $p_1 + p_2 + p_3 + p_4 \sim p_5 + p_6 + p_7 + p_8$ is not invariant under the symmetry between 
 $\{ p_{3}, p_{4}\}$ and $\{ p_{5}, p_{6}\}$, leading to a contradiction. 
\end{proof}

\begin{remark}
\label{rem:trigonal-component}
We claim that $\overline{BN}^1_3$ is nonempty, i.e. $\pi^{-1}BN^1_3$ consists of the main component $\wt{BN}^1_3$ as well as some boundary divisors. For example, take a pencil of plane cubics and mark a base point as $p_8$. Attach it to $\bbP^1$ at another base point 
 and mark seven general points in $\bbP^1$ as  $p_1,\ldots, p_7$. We obtain a curve $C$ moving in $\delta_{0 ; \{1,\ldots, 7 \}}$ with the following intersection numbers: 
$$ C\cdot \lambda = 1, $$
$$ C\cdot \delta_{0; \{1, \ldots, 7\}} = -1, $$
$$ C\cdot \delta_{0; S} = 0, \quad S\neq  \{1, \ldots, 7\}. $$
It follows that $C\cdot \pi^{*} BN^1_3 < 0$, hence $\pi^{*}BN^1_3$ contains 
$\Delta_{0; S}$ for any $|S| = 7$. It would be interesting to calculate directly the class of the main 
component $\wt{BN}^1_3$. 
\end{remark}


\section{Pulling back the $d$-gonal divisor}
\label{sec:d-gonal}

Consider $\pi: \BM_{1,4d-4} \to \BM_{g}$, where $g = 2d-1$. Let $BN^1_d\subset \BM_g$ be the Brill-Noether divisor parameterizing $d$-gonal curves. 
By \cite{HarrisMumfordKodaira}, it has class 
\begin{eqnarray}
\label{eq:d-gonal}
BN^1_d & = & c \left((2d+2)\lambda - \frac{d}{3}\delta_0 - \sum_{i=1}^{d-1}i (2d-1-i)\delta_i\right) \\ \nonumber 
                & = & c \left((2d+2)\lambda - \frac{d}{3}\delta - \sum_{i=1}^{d-1}  \Big(i (2d-1-i) - \frac{d}{3}\Big)\delta_i \right), 
\end{eqnarray}
where $c = \frac{3(2d-4)!}{d!(d-2)!}$.  In this section we will show that the main component of $\pi^{*}BN^{1}_{d}$ is extremal.

Consider the surface $S = E\times \bbP^1$, where $E$ is a smooth genus one curve. Let $\pi_0$ and $\pi_1$ be the projections to $\bbP^1$ and $E$, respectively. 
We have 
$$ \Pic(S) \cong \bbZ[e] \oplus \pi_1^{*}\Pic(E), $$
where $e = \pi_0^{*} \OO_{\bbP^1}(1)$ represents the genus one fiber class. Let $D$ be a divisor of degree $d$ on $E$. 
Projecting a curve in the linear system $|e + \pi_1^{*}D|$ via $\pi_0$ admits a degree $d$ cover of $\bbP^1$. Take $2d-2$ general genus one fibers $E_1, \ldots, E_{2d-2}$ and fix a general point $p_{2k-1}\in E_k$ for $k=1,\ldots, 2d-2$. 
Since $\dim |e + \pi_1^{*}D| = 2d-1$, we obtain a pencil $B$ of curves $C_b$ in $S$ that pass through all the fixed $p_{2k-1}$. 
Denote by $p_{2k}$ an intersection point of $C_b$ with $E_i$ other than $p_{2k-1}$. Then $(C_b; p_1, \ldots, p_{4d-4})$ 
is a genus one curve with $4d-4$ marked points, hence $B$ can be viewed as a curve in $\BM_{1,4d-4}$. Since $p_{2k-1}$ and $p_{2k}$ are 
both contained in $E_k$, projecting to $\bbP^1$ realizes $\pi(E, p_1,\ldots, p_{4d-4})$ as a $d$-gonal curve. 

Note that $C_b$ has the same $j$-invariant 
as that of $E$, because it admits a one-to-one map to $E$ via $\pi_1$. It implies that 
$$ B\cdot \lambda = B\cdot \delta_{\irr} = 0. $$

Since $(e + \pi_1^{*}D)^2 = 2d$, besides $p_1, p_3, \ldots, p_{4d-5}$ there are two other base points $q_1$ and $q_2$ in the pencil $B$. Hence there are 
$2d$ singular curves in $B$, each of which consists of a genus zero fiber passing through $s = p_{2k-1}$ or $s = q_i$ union a (unique) genus one curve 
in $|e + \pi_1^{*}(D-s)|$ passing through the remaining base points.

\begin{example}
\label{ex:d=3}
Consider $d = 3$ and $\pi: \BM_{1,8} \to \BM_{5}$. Make a base change of degree $2^4$, still denoted by $B$, so that we can distinguish 
the two remaining intersection points of $E_{k}$ and $C_b$ besides $p_{2k-1}$ for $k=1,2,3,4$. See Figure~\ref{fig:trigonal} for the configuration. 

\begin{figure}[htb]
    \centering
    \def\svgwidth{200pt}
     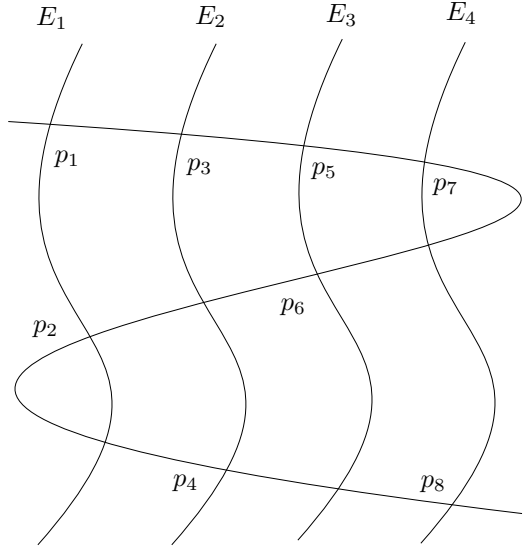
        \caption{\label{fig:trigonal} A genus one curve marked at $p_1,\ldots, p_8$ }
    \end{figure}

If $p_{2k}$ coincides with $p_{2k-1}$, it contributes 
$2\cdot 2^3 = 2^4$ to $B\cdot \delta_{0; \{2k-1, 2k\}}$ (due to the base change). For a singular curve in $B$, if $s = q_i$ for $i=1,2$, it contributes 
$1$ to $\delta_{0; S}$ where $S\subset \{2,4,6,8\}$. If $s = p_{2k-1}$, then it contributes $2$ to $\delta_{0;S\cup \{2k-1\}}$ for $S\subset \{2,4,6,8\}\backslash\{2k\}$, where the number $2$ comes from the choice of $p_{2k}$. 

We thus obtain that the only nonzero intersections of $B$ with boundary divisors are as follows: 
$$ B\cdot \delta_{0;\{p_{2k-1}, p_{2k}\}} = 16, $$
$$ B\cdot \delta_{0;S} = 2, \quad S\subset \{2,4,6,8\}, $$
$$ B\cdot \delta_{0;S\cup \{2k-1\}} = 2, \quad S\subset \{2,4,6,8\}\backslash\{2k\}. $$
It follows that 
$$ B\cdot \pi^{*}BN^1_3 = -128 + 2(1\cdot 6 + 2\cdot 4 + 3\cdot 1) + 2\cdot 4\cdot (1\cdot 3 + 2\cdot 3 + 3\cdot 1) 
= 2 > 0. $$
\end{example}

Nevertheless, for $d\geq 4$ the above intersection number turns out to be negative. 

\begin{lemma}
\label{lem:d-gonal}
In the above setting, $B\cdot \pi^{*}BN^1_d < 0$ for $d\geq 4$. 
\end{lemma}

\begin{proof}
The calculation is similar to Example~\ref{ex:d=3}. Make a degree $(d-1)^{2d-2}$ base change, so that we can distinguish the marked point $p_{2k}$ out of the remaining $d-1$ intersection points of $E_k$ and $C_b$ for $k=1,\ldots, 2d-2$. If 
$p_{2k}$ coincides with $p_{2k-1}$, it contributes $(d-1)(d-1)^{2d-3} = (d-1)^{2d-2}$ 
to $B\cdot \delta_{0; \{2k-1, 2k\}}$. For a singular curve in $B$, if it passes through $s = q_i$ for $i=1$ or $2$, both of them contribute $(d-2)^{2d-2-|S|}$ to $\delta_{0; S}$ for $S\subset \{2,4,\ldots, 4d-4\}$, due to the choice of the marked points in $\{2,4,\ldots, 4d-4\}\backslash S$. If $s = p_{2k-1}$, it contributes 
$(d-1)(d-2)^{2d-3-|S|}$ 
to $\delta_{0;S\cup \{2k-1\}}$ for $S\subset \{2,4,\ldots, 4d-4\} \backslash \{ 2k\}$, due to the choice of $p_{2k}$ and the 
 marked points in $\{2,4,\ldots, 4d-4\} \backslash (\{ 2k\}\cup S)$. 

We thus obtain that
$$ B\cdot \delta_{0;\{p_{2k-1}, p_{2k}\}} = (d-1)^{2d-2}, $$
$$ B\cdot \delta_{0;S} = 2(d-2)^{2d-2-|S|}, \quad S\subset \{2,4,\ldots,4d-4\}, $$
$$ B\cdot \delta_{0;S\cup \{2k-1\}} = (d-1)(d-2)^{2d-3-|S|}, \quad S\subset \{2,4,\ldots, 4d-4\}\backslash\{2k\}. $$
In $\pi^{*}BN^1_d/c$, the coefficients of $\delta_{0;\{2k-1, 2k\}}$, of $\delta_{0;S}$ for 
$S\subset \{2,4,\ldots,4d-4\}$ and of $\delta_{0;S\cup \{2k-1\}}$ for $S\subset \{2,4,\ldots, 4d-4\}\backslash\{2k\}$ 
are 
$$ 2 - \frac{4}{3}d,  \quad \frac{d}{3} (|S|-1), \quad \frac{d}{3} |S|, $$
respectively, by Proposition~\ref{prop:pullback} and \eqref{eq:d-gonal}. It follows that 
\begin{eqnarray*}
 \frac{1}{c}B\cdot \pi^{*}BN^1_d & = & \Big(2-\frac{4}{3}d\Big) 2(d-1)^{2d-1}  + \frac{d}{3} \sum_{s=1}^{2d-2} 2(d-2)^{2d-2-s}(s-1){2d-2 \choose s} \\
& &  + \frac{d}{3}(2d-2) \sum_{s=1}^{2d-3} (d-1)(d-2)^{2d-3-s}s{2d-3\choose s} \\
& = &  \Big(2-\frac{4}{3}d\Big) 2 (d-1)^{2d-1} + \frac{2d}{3}\Big((d-1)^{2d-2} + (d-2)^{2d-2}\Big) \\
& & + \frac{2d}{3} (2d-3)(d-1)^{2d-2} \\
& = & \frac{2}{3}\Big( d(d-2)^{2d-2} - 2 (d-3)(d-1)^{2d-1} \Big). 
\end{eqnarray*}
It is easy to check that $B\cdot \pi^{*}BN^1_d < 0$ for $d \geq 4$. 
\end{proof}

Taking away possible boundary components, denote by $\wt{BN}^1_d$ the main component of $\pi^{-1}BN^1_d$. 

\begin{lemma}
\label{lem:d-gonal-irreducible}
The main component $\wt{BN}^1_d$ is irreducible. 
\end{lemma}

\begin{proof}
Let $U\subset \wt{BN}^1_d$ be the open dense subset parameterizing simply branched degree $d$, genus one covers 
$f: E\to \bbP^1$ with a general choice of $m$ pairs of points $(p_1, p_2), \ldots, (p_{2m-1}, p_{2m})$ in $E$, where $m = 2d-2$ and $f(p_{2j-1}) = f(p_{2j}) = q_j \in \bbP^1$. It suffices to show that $U$ is irreducible. 

Let $b_1, \ldots, b_{2d}\in \bbP^1$ be the set of branch points of $f$. The branch data $\phi_1, \ldots, \phi_{2d}$ associated to the branch points can be arranged as 
$$(1,2), (1,2), (1,2), (1,2), (2, 3), (2, 3), (3, 4), (3, 4), \ldots, (d-1, d), (d-1, d),$$ 
see e.g. \cite[page 100]{EEHS}. In other words, the monodromy induced by a closed, suitably oriented loop centering around $b_i$ is the permutation $\phi_i \in S_d$. 

The choice of $p_{2j-1}$ and $p_{2j}$ amounts to choosing two distinct numbers $a_{2j-1}$ and $a_{2j}$ out of 
$\{ 1, \ldots, d \}$, i.e. specifying two of the $d$ sheets of $f$ over $q_j$. Without loss of generality, assume that 
$a_{2j-1} < a_{2j}$. Vary $q_j$ along the loops centering around the $b_i$ with branch data $(a_{2j-1}-1, a_{2j-1}), (a_{2j-1}-2, a_{2j-1}-1), 
\ldots, (1, 2)$ successively. When $q_j$ comes back to the original position, $a_{2j-1}$ is transformed to $1$. Next, 
vary $q_j$ along the loops centering around the $b_i$ with branch data $(a_{2j}-1, a_{2j}), (a_{2j}-2, a_{2j}-1), 
\ldots, (2, 3)$ successively. As a result, $a_{2j}$ is transformed to $2$. Finally, if we vary $q_j$ along the loop around $b_1$ with branch datum $(1,2)$, the ordered pair $(1, 2)$ is transformed to $(2,1)$. Meanwhile, the other pairs 
$(a_{2k-1}, a_{2k})$ are unchanged since we did not vary $q_k$ for $k\neq j$. Carrying out this process for 
$j=1,\ldots, m$ one by one, eventually all the pairs $(a_{2j-1}, a_{2j})$ can be transformed to $(1,2)$. 

Let $W$ be the open dense subset of the Hurwitz space of degree $d$, genus one, simply branched covers of $\bbP^1$. The above process implies that the monodromy 
of $U\to W$ is transitive, where the map is finite of degree $(d(d-1))^{m}$ forgetting the $p_i$. 
Since $W$ is irreducible, see \cite{Hurwitz, EEHS}, it follows that $U$ is irreducible.  
\end{proof}

\begin{theorem}
\label{thm:d-gonal}
For $d\geq 4$, $\wt{BN}^1_d$ spans an extremal ray of $\BEff(\BM_{1,4d-4})$ and it is different from the ray of any $D_{\bf a}$. 
\end{theorem}

\begin{proof}
By Lemma~\ref{lem:d-gonal}, $B\cdot \pi^{*}BN^1_d < 0$ and it is not entirely contained in the boundary of $\BM_{1,4d-4}$, hence $B\cdot \wt{BN}^1_d < 0$. Moreover, by the construction of $B$, it is a moving curve in $\wt{BN}^1_d$. It follows that $\wt{BN}^1_d$ spans an extremal ray of $\BEff(\BM_{1,4d-4})$. 
The same argument 
as in the proof of Theorem~\ref{thm:trigonal-extremal} shows that $\wt{BN}^1_d$ and $D_{\bf a}$ span different extremal rays of $\BEff(\BM_{1,4d-4})$. 
\end{proof}

\section{Pulling back the Gieseker-Petri divisor}
\label{sec:GP}

Consider the gluing map $\pi: \BM_{1,6}\to \BM_4$. Let $GP$ denote the Gieseker-Petri divisor in $\BM_{4}$. The divisor $GP$ has several geometric interpretations. Here we view it as the closure of the locus in $\BM_4$
 parameterizing genus four curves whose canonical images are contained in a quadric cone in $\bbP^3$. By \cite[Theorem 2]{EisenbudHarrisKodaira}, it has class  
\begin{eqnarray*}
GP & = & 34\lambda - 4\delta_{\irr} - 14\delta_{1} - 18 \delta_2 \\
      & = & 34\lambda - 4\delta - 10\delta_1 - 14\delta_2. 
\end{eqnarray*}
By Proposition~\ref{prop:pullback}, we obtain: 
\begin{eqnarray}
\label{eqn:gp}
 \pi^{*}GP = 10\lambda + 4\sum_{S\not\in \cup \Lambda_i} (|S| - 1)\delta_{0;S} + 10 \delta_{0;\{1,2,3,4,5,6\}} \\
- 6 \big(\delta_{0;\{1,2\}} +  \delta_{0;\{3,4\}} +  \delta_{0;\{5,6\}}\big) - 2 \big(\delta_{0;\{1,2,3,4\}} + \delta_{0;\{1,2,5,6\}}  + \delta_{0;\{3,4,5,6\}}\big)\nonumber. 
\end{eqnarray}

In principle, $\pi^{-1}GP$ may contain some boundary components of $\BM_{1,6}$, as we saw in Remark~\ref{rem:trigonal-component}. Denote by $\wt{GP}$ the main component of $\pi^{-1}GP$, i.e. 
$\wt{GP}$ is the closure of $(\pi^{-1}GP)\cap \MM_{1,6}$. 

The main result of this section is that $\wt{GP}$ spans an extremal ray of $\BEff(\BM_{1,6})$. 
We prove it by constructing a moving curve $B \subset \wt{GP}$ such that $B \cdot \wt{GP} < 0$.  

We begin by observing that the general element $(E; p_{1}, ... ,p_{6}) \in \wt{GP}$ is obtained as the (marked) normalization of a $3$-nodal elliptic curve $C \subset \bbF_{2}$ where $\bbF_{2}$ is the second Hirzebruch surface.  Let $\sigma$ be the directrix class on $\bbF_{2}$ and let $\tau$ be the section class satisfying $\tau \cdot \sigma = 0$. The class of $C$ is $3\tau$.  We will essentially vary the $3$-nodal curve $C$.  However, it will also turn out to be necessary to ``vary" the ambient surface $\bbF_{2}$.  Let us first construct the variation of the ambient surface. 

Let $Q = \bbP^1 \times \bbP^1$, with projections $\pi_{1}$ and $\pi_{2}$.  Let $f_{1}$ and $f_{2}$ denote the classes 
of the fibers of $\pi_{1}$ and $\pi_{2}$, respectively. Let $\OO_{Q}(m,n)$ denote the line bundle associated to the divisor class $mf_{1} + nf_{2}$. Let 
$V =\OO_{Q} \oplus \OO_{Q}(1,2)$
be a rank two vector bundle on $Q$. Consider the $\bbP^1$-bundle 
\[Y = {\rm Proj}(V),\] 
with its projection $p : Y \to Q.$ Let $\phi: Y \to \bbP^1$ be the composite $\pi_{1} \circ p$. Then $\phi$ is a non-trivial family of $\bbF_{2}$'s.  The threefold $Y$ will essentially serve as the ambient space of the varying family of $3$-nodal elliptic curves. 

Let us understand the geometry of $Y$. The Picard group of $Y$ is $\bbZ[\zeta,f_{1},f_{2}]$, where here, and in everything that follows, we suppress pullbacks $p^{*}$'s from notation. The projection $p$ has a distinguished section  $\Sigma \subset Y$, which can be described as the union of the $(-2)$-curves in the $\bbF_{2}$'s.  The class of $\Sigma$ is 
\[\Sigma = \zeta - f_{1} - 2f_{2}.\] 
We also select a disjoint section $\Pi \subset Y$ having divisor class 
\[\Pi = \zeta.\]  
In terms of the family $\phi : Y \to \bbP^1$, the section $\Pi$ provides a family of sections complementary to the directrices of the varying $\bbF_{2}$ fibers. The projection $p$ obviously restricts to an isomorphism on the section $\Pi$, so it makes sense to refer to the rulings of $\Pi$ by $f_{1}$ and $f_{2}$ as well. 


We will now blow up $Y$ along the union of three disjoint curves. Restricted to each fiber $\bbF_2$, this amounts to blowing up three points corresponding to the nodes of the desired $3$-nodal curve $C$.  

Let $Z_{1}, Z_{2}$ and $Z_{3}$ be three disjoint curves of the ruling class $f_{2}$ on $\Pi$. Let $Z = Z_{1} \cup Z_{2} \cup Z_{3}$ and let $X = \Bl_{Z}Y$ with the blow down map $\beta : X \to Y$.  Denote by $\varphi: X \to \bbP^1$ the composite $\phi\circ\beta$. Then $\varphi$ is a family of surfaces $S_{t}$, each being the blow up of $\bbF_{2}$ at three points $z_{1}(t), z_{2}(t)$ and $z_{3}(t)$ for $t \in \bbP^1$.  The Picard group of $X$ is 
$ \bbZ[\zeta, f_{1}, f_{2}, e_{1}, e_{2}, e_{3}]$ where the $e_{i}$ are the classes of the respective exceptional divisors $E_{i}$.  Let 
$e = e_{1} + e_{2} + e_{3}$ and $E = E_{1} \cup E_{2} \cup E_{3}$. Since 
$N_{Z_i/Y}$ is isomorphic to $\OO \oplus \OO(1)$, each exceptional divisor $E_{i}$ is isomorphic to the first Hirzebruch surface $\bbF_{1}$. 

Now we consider the divisor class $3\zeta + (a-2)f_{1} - 2e$ 
on $X$, where $a \gg 0$. 
Let $l$ be the class of a line and let $r$ be  the class of a ruling on $E_i\cong \bbF_{1}$. Note that 
$\zeta|_{E_i} = f_1|_{E_i} = r$ and $e_{i}|_{E_i} = r-l$. Therefore, 
$$ \big(3\zeta + (a-2)f_{1} - 2e\big) |_{E_i} =  2l + (a-1)r. $$

\begin{lemma}\label{gen_D} 
Let $D$ be a general divisor with divisor class $3\zeta + (a-2)f_{1} -2e$ on  $X$, where $a\gg 0$.
\begin{enumerate}[label={\upshape(\roman*)}]
\item $D$ is a smooth surface in $X$. Under the map $\varphi : D \to \bbP^1$, the fibers have genus one and the singular fibers are at worst nodal. 

\item The curves $C_{i} = D \cap E_{i}$ provide smooth $2$-sections to the family $\varphi : D \to \bbP^1$ for $i=1,2,3$. Moreover, the induced double cover $\varphi: C_{i} \to \bbP^1$ has $2a$ ramification points. 

\item There are exactly $(a+1)$ singular fibers in the family $\varphi : D \to \bbP^1$ having rational tails. Every rational tail meets $C_{i}$ once for each $i$. 
\end{enumerate}
\end{lemma}

\begin{proof}
 Let $S$ be a fiber of $\varphi: X \to \bbP^1$.  So $S = \Bl_{3}\bbF_{2}$.  We first show that the restriction map \[r : H^{0}(X, \OO_{X}(D)) \to H^{0}(S, \OO_{S}(D))\] is surjective.  By the long exact sequence, it is enough to show that \[h^{1}(X, \OO_{X}(D-S) = 0.\]  

By the Leray spectral sequence for $\varphi$, $h^{1}(X, \OO_{X}(D-S)) = 0$ if 
$$h^{0}(R^{1}\varphi_{*}(\OO_{X}(D-S)) = h^{1}(\varphi_{*}\OO_{X}(D-S))=0. $$  
The sheaf $R^{1}\varphi_{*}\OO_{X}(D-S)$ vanishes, because the line bundle class $3\tau - 2e$ on $S = \Bl_{3}\bbF_{2}$ has no higher cohomology.  Furthermore, by push-pull, we have 
$$\varphi_{*}\OO_{X}(D-S)\cong \varphi_{*}(\OO_{X}(3\zeta - 2e))\otimes \OO_{\bbP^{1}}(a-3),$$ 
which has vanishing higher cohomology for $a \gg 0$. 
Therefore, $r$ is surjective. The divisor class $3\tau - 2e$ on $\Bl_{3}\bbF_{2}$ is base point free. By adjunction, curves parameterized by the linear system $|3\tau - 2e|$ have genus one. Furthermore, the locus of curves with worse than nodal singularities is of codimension $2$ in $|3\tau - 2e|$.  Therefore, (i) follows from Bertini's theorem.  

In exactly the same way, one can show that the restriction map $H^{0}(\OO_{X}(D)) \to H^{0}(\OO_{E_{i}}(D))$ is surjective when $a\gg 0$. The adjunction formula applied to the curve class $2l + (a-1)r$ implies that $C_i$ 
has genus $a-1$, thus proving (ii).

The proper transform $\Pi'$ of $\Pi$ in $X$ is isomorphic to $\Pi$ under the blow down map $\beta: X\to Y$.  The intersection $D \cdot \Pi'$ has divisor class $(a+1)\wt{f}_{1}$ in $\Pi'$, where $\wt{f}_{1}$ is the ruling class 
corresponding to $f_1$ under the isomorphism $\Pi' \cong \Pi$. It gives rise to $(a+1)$ disjoint curves that are precisely the rational tails in the family $\varphi: D \to \bbP^1$. Hence (iii) follows right away.
\end{proof}
 
 Our next goal is to compute the degree of $\delta_{\irr}$ restricted to $\varphi: D\to \bbP^1$, i.e. the number of nodal fibers in this family which do not have a rational tail.  It is equal to $12 \lambda$ by \eqref{eqn:delta}.  Before doing so, we gather some relevant intersection products in the Chow ring of $X$.
 
 \begin{lemma}\label{intersections}
 The following intersection numbers hold in the Chow ring of $X$: 
 \begin{eqnarray}\label{eq_int}
 & f_{i}^{2} = f_{2} \cdot e_{i} = \Sigma \cdot \Pi = \Sigma \cdot e_{i} =  0, \nonumber \\
 & \Pi^{3} = \Sigma^{3} = 4, \nonumber \\
 &  e_{i}^{3} = -1, \nonumber \\
 & e_{i}^2 \cdot f_{1} = e_{i}^2 \cdot \Pi = -1, \\ 
 &  \Sigma^2 \cdot f_{1} = -\Pi^2 \cdot f_{1} = -2, \nonumber \\
 &  \Sigma^2 \cdot f_{2} = -\Pi^2 \cdot f_{2} = -1, \nonumber \\
 &   \Sigma \cdot f_{1} \cdot f_{2} = \Pi \cdot f_{1} \cdot f_{2} = 1. \nonumber
 \end{eqnarray}
 
 Furthermore, let $l_{i}$ and $r_{i}$ be the line and ruling classes of $E_{i} \cong \bbF_{1}$.  Then $E_{i}\cdot {E_{i}} = r_{i} - l_{i}$ and $\Pi \cdot E_{i} = f_{1} \cdot E_{i} = r_{i}$. 
 \end{lemma}
 
 \begin{proof}
We will prove only those which are not immediately clear.  First, we see that \[\Pi^2 = (f_{1} + 2f_{2})\cdot \Pi\] which follows from the general fact that \[\zeta^{r} = \sum_{i = 1}^{r} (-1)^{i+1}c_{i}(V) \cdot \zeta^{r-i}\] for any projective bundle ${\rm Proj\,} V$, where $V$ is a rank $r$ vector bundle and $\zeta$ is the universal line bundle class of ${\rm Proj\,} V$. Here in our case $\Pi$ has class $\zeta$. From this, one easily derives the formulas $\Pi^{3} = \Sigma^{3} = 4,$ and all other formulas involving only $\Pi, \Sigma, f_{1}$ and $f_{2}$.

We now deal with the intersections involving the exceptional divisors $E_{i}$. If $W\subset G$ is a smooth subscheme of a variety $G$ and if ${\wt G} = \Bl_{W}G$, then the exceptional divisor $E \subset {\wt G}$ is isomorphic to ${\rm Proj \,} (I_{W}/I^{2}_{W})$, and $\OO_{E}(E) = \OO_{E}(-1)$ has class $r-l$.  In our setting, the conormal bundle of $Z_{i}$ in $Y$ is $\OO_{Z_{i}} \oplus \OO_{Z_{i}}(-1)$.  Therefore, $E_{i}$ is abstractly isomorphic to $\bbF_{1}$, and $\OO_{E_{i}}(1)$ is the line bundle associated to the directrix $\sigma \subset \bbF_{1}$. The class of the directrix is also given by $l_{i} - r_{i}$, and therefore $E_{i} \cdot E_{i} = r_{i} - l_{i}$.  Furthermore, since $\Pi \cdot Z_{i} = f_{1} \cdot Z_{i} =  1$ in $Y$, we conclude that $\Pi \cdot E_{i} = f_{1} \cdot E_{i} = r_{i}$.   

The remaining intersection products follow from those explained above.
 \end{proof}
 
 \begin{lemma}\label{lambda}
 For the family $\varphi : D \to \bbP^1$, $\delta_{\irr} = 12\lambda = 12(a-1)$.
 \end{lemma}

\begin{proof}
First observe that $\lambda = \chi(\OO_{D})$. Secondly, by Noether's formula we know that 
$$12\chi(\OO_{D}) = K_{D}^2 + c_{2}(T_{D}).$$ 
Since the fibers of $\varphi$ have genus one, the intersection number $K_{D}^2$ is $-(a+1)$, where each rational tail contributes $-1$. 

So we need only compute $c_{2}(T_{D})$. Using the exact sequence 
\[0 \to T_{D} \to T_{ X}|_{D} \to N_{D/{X}} \to 0, \] 
we see that 
\begin{eqnarray}\label{c2T}
c_{2}(T_{D}) & = & c_{2}(T_{X}|_{D}) - c_{1}(T_{D}) \cdot c_{1}(N_{D/X}) \nonumber \\ 
& = & c_{2}(T_{X}|_{D}) + K_{D}\cdot c_{1}(N_{D/X}) \nonumber \\
& = & c_{2}(T_{X})\cdot D + (D + K_{X})\cdot D^2. 
\end{eqnarray}
Equality \eqref{c2T} indicates that we need only compute $c_{2}(T_{X})$, and then we can proceed by using the intersection formulas provided in 
Lemma~\ref{intersections}. Consider the sequence 
\begin{equation}\label{blowup}
0 \to T_{X} \to \beta^{*}T_{Y} \to T_{E/Z}(E) \to 0, 
\end{equation}
from which we conclude 
\begin{equation}\label{c2TX}
c_{2}(T_{X}) = i_{*}[-\omega_{E/Z}] + E \cdot K_{X} + \beta^{*}c_{2}(T_{Y})
\end{equation}
where $i: E \to X$ is the inclusion map. For \eqref{c2TX}, we can use the Grothendieck-Riemann-Roch formula for the inclusion $i$ to deduce 
$$c_{1}(T_{E/Z}(E)) = E, $$ 
$$c_{2}(T_{E/Z}(E)) = i_{*}[c_{1}(\omega_{E/Z})] = r-2l.$$ 
Using the exact sequence 
\[ 0 \to \OO_{Y}(\Sigma+\Pi) \to T_{Y} \to p^{*}(T_{Q}) \to 0, \] we obtain
\begin{equation}\label{rel_tangent}
c_{2}(T_{Y}) = -K_{Y}\cdot (\Sigma + \Pi) + 4f_{1} \cdot f_{2}.
\end{equation}
   Finally, we can compute the quantity $c_{2}(T_{D})$ using the equality \eqref{c2T} and the intersection numbers given in Lemma~\ref{intersections}.   The end result is
\begin{equation}\label{delta}
c_{2}(T_{D}) = 13a-11. 
\end{equation}
Therefore, $12\lambda = -(a+1) + 13a-11 = 12a-12$ as claimed.
\end{proof}

There are three types of singular fibers in the genus one family $\varphi: D \to \bbP^1$: 
\begin{enumerate}[label={\upshape(\roman*)}]

\item Irreducible nodal fibers. 

\item Fibers having a rational tail. The rational tail meets each $2$-section $C_{i}$ once.

\item Two $\bbP^1$'s joined at two nodes. One component will be the directrix $\sigma \subset \Bl_{3}\bbF_{2}$. The other component will intersect each $C_{i}$ twice.
\end{enumerate}

In order to obtain a family of genus one curves with six ordered marked points, we need to perform a base change $T \to \bbP^1$ of degree $2^3 = 8$ to distinguish $p_{2k-1}$ and $p_{2k}$ that are glued as a node for $k=1,2,3$. 
We name the sections so that the marked points $\{p_{2k-1},p_{2k}\}$ correspond to the $2$-section $C_{k}$. 
Then the curve $T$ carries over it a family of genus one curves with six ordered marked points.  

Using Lemmas~\ref{gen_D}, \ref{lambda} and the above analysis of singular fibers of $\varphi$, we obtain the following intersection numbers: 
\begin{enumerate}[label={\upshape(\roman*)}]
\item $T \cdot \lambda = 8(a-1)$. 
\item $T \cdot \delta_{0;\{i,j,k\}} = a+1$ for each triple $(i,j,k) \in \{1,2\} \times \{3,4\} \times \{5,6\}$. 
\item $T \cdot \delta_{0;\{2k-1,2k\}} = 4 \cdot (2a) = 8a$ for $k = 1, 2, 3$. 
\end{enumerate}

All intersections of $T$ with other boundary divisors of $\BM_{1,6}$ are trivial. Plugging these numbers into the expression of $\pi^{*}(GP)$ in \eqref{eqn:gp} yields: 
\[T\cdot (\pi^{*}GP) = -16.\] 

\begin{theorem}
\label{thm:GP}
The divisor $\wt{GP}$ is extremal in $\BEff(\BM_{1,6})$. For $n\geq 6$, the pullback of 
$\wt{GP}$ to $\BM_{1,n}$ spans an extremal ray of $\BEff(\BM_{1,n})$, which is different from 
the rays spanned by $D_{\bf a}$'s. 
\end{theorem}

\begin{proof}
Since $T$ parameterizes curves in the main component ${\wt{GP}}$, we conclude that 
$$T \cdot {\wt{GP}} \leq T\cdot (\pi^{*}GP) < 0.$$  
Furthermore, the family $T$ is evidently moving in $\wt{GP}$, hence $\wt{GP}$ is an extremal divisor. Using the symmetry between the three pairs of marked points $\{p_{2k-1}, p_{2k}\}$ for $k=1,2,3$, the same argument as in the proof of Theorem~\ref{thm:trigonal-extremal} shows the second part of the claim. 
\end{proof}



\begin{thebibliography}{Ab123}

\bibitem[AC]{ArbarelloCornalba}
   E.~Arbarello, and M.~Cornalba, The {P}icard groups of the moduli spaces of curves,
   {\em Topology} {\bf 26} (1987), no. 2, 153--171.

\bibitem[CT]{CastravetTevelev}
	A.-M.~Castravet,  and J.~Tevelev, Hypertrees, projections and moduli of stable rational curves,
	{\em J. Reine Angew. Math.} {\bf 675} (2013), 121--180. 
	
	
\bibitem[CC]{ChenCoskun}
D.~Chen, and I.~Coskun, Extremal effective divisors on $\BM_{1,n}$, 
{\em Math. Ann.} {\bf 359} (2014), 891--908. 

\bibitem[EEHS]{EEHS}
D.~Eisenbud, N.~Elkies, J.~Harris, and R.~Speiser, On the Hurwitz scheme and its monodromy, 
{\em Compos. Math.} {\bf 77} (1991), no. 1, 95--117. 


\bibitem[EH]{EisenbudHarrisKodaira}
	D.~Eisenbud, and J.~Harris, 
	The Kodaira dimension of the moduli space of curves of genus $\geq 23$,
	{\em Invent. Math.}
	{\bf 90} (1987), no. 2, 359--387.
	
\bibitem[F]{FarkasKoszul}
  G.~Farkas, Birational aspects of the geometry of $\BM_g$,
  {\em Surv. Differ. Geom.}
   {\bf 14} (2009), 57--110. 

\bibitem[Ha]{HarrisKodaira}
	J.~Harris, 
	On the {K}odaira dimension of the moduli space of curves {II}. The even-genus case, 
	{\em Invent. Math.}
	 {\bf 75} (1984), no. 3, 437--466.

\bibitem[HMo]{HarrisMorrison}
   J.~Harris, and I.~Morrison, {\em The moduli of curves}, Graduate Texts in Mathematics {\bf 187}, Springer, New York, 1998.

\bibitem[HMu]{HarrisMumfordKodaira} 
	J.~Harris, and D.~Mumford, 
	On the Kodaira dimension of the moduli space of curves,
	{\em Invent. Math.}
	 {\bf 67} (1982), no. 1, 23--88.
	 
\bibitem[Hu]{Hurwitz}
A.~Hurwitz, Ueber Riemann'sche Fl\"achen mit gegebenen Verzweigungspunkten	(German), 
{\em Math. Ann.} {\bf 39} (1891), no. 1, 1--60. 

\bibitem[K]{Kollar}
   J.~Koll\'ar, {\em Rational curves on algebraic varieties}, 
   Springer-Verlag, Berlin, 1996.

\bibitem[Lo]{LoganKodaira}
	A.~Logan, 
	The Kodaira dimension of moduli spaces of curves with marked points,
	 {\em Amer. J. Math.} {\bf 125} (2003), no. 1, 105--138.
	 
\bibitem[O]{OpieExtremal}
         M.~Opie, 	 
         Extremal divisors on moduli spaces of rational curves with marked points, 
         arXiv:1309.7229. 
	 
\bibitem[V]{Vermeire} 
	P.~Vermeire,  A counterexample to Fulton's conjecture on $\overline{M}_{0,n}$,
	{\em J. Algebra} {\bf 248} (2002),  no. 2, 780--784.

\end{thebibliography}
\end{document}